\documentclass[submission]{FPSAC2018}

\usepackage{comment}
\usepackage{latexsym,enumerate}
\usepackage{graphicx,pict2e}
\usepackage{multirow,multicol,mathtools}
\usepackage{tikz}
\usepackage{ytableau}
\usepackage[all,cmtip]{xy}
\usepackage{young}

\newlength\circlesize
\allowdisplaybreaks[1]

\usepackage[cmtip,all]{xy}
\newcommand{\longsquiggly}{\xymatrix{{}\ar@{~>}[r]&{}}}

\newcommand{\ShST}{\mathrm{ShST}}

\newcommand{\std}{\mathrm{std}}

\newcommand{\rectify}{\mathrm{rect}}

\newcommand{\height}{\mathrm{ht}}

\newcommand{\B}{\mathcal{B}}

\newcommand{\wt}{\mathrm{wt}}

\newtheorem{lemma}{Lemma}
\newtheorem{theorem}[lemma]{Theorem}

\theoremstyle{definition}
\newtheorem{example}[lemma]{Example}
\newtheorem{definition}[lemma]{Definition}
\newtheorem{remark}[lemma]{Remark}

\numberwithin{equation}{section}
\numberwithin{figure}{section}
\numberwithin{table}{section}
\numberwithin{lemma}{section}

\newcommand{\defn}[1]{{\bf #1}}

\newcommand{\east}[1]{\ensuremath{\xrightarrow{\ #1\ }}}
\newcommand{\west}[1]{\ensuremath{\xleftarrow{\ #1\ }}}
\newcommand{\north}[1]{\ensuremath{\big \uparrow \!\text{\raisebox{.1ex}{\scriptsize $#1$}}}}
\newcommand{\south}[1]{\ensuremath{\big \downarrow \!\text{\raisebox{.1ex}{\scriptsize $#1$}}}}

\newcommand{\stepnorth}[2]{\vector(0,1){.92}\put(-.23,.4){\scriptsize$#1$}\put(0,1){#2}}
\newcommand{\stepnorthA}[2]{\vector(0,1){.92}\put(.05,.4){\scriptsize$#1$}\put(0,1){#2}}
\newcommand{\stepnorthshiftW}[2]{\put(-.08,0){\vector(0,1){.95}\put(-.23,.4){\scriptsize$#1$}}\put(0,1){#2}}

\newcommand{\stepsouth}[2]{\vector(0,-1){.92}\put(.05,-.6){\scriptsize$#1$}\put(0,-1){#2}}

\newcommand{\stepsouthshiftE}[2]{\put(.08,0){\vector(0,-1){.92}\put(.05,-.6){\scriptsize$#1$}}\put(0,-1){#2}}
\newcommand{\stepsouthshiftW}[2]{\put(-.08,0){\vector(0,-1){.92}\put(-.23,-.6){\scriptsize$#1$}}\put(0,-1){#2}}
\newcommand{\stepeast}[2]{\vector(1,0){.92}\put(.4,-.25){\scriptsize$#1$}\put(1,0){#2}}
\newcommand{\stepeastA}[2]{\vector(1,0){.92}\put(.4,.05){\scriptsize$#1$}\put(1,0){#2}}

\newcommand{\stepwest}[2]{\vector(-1,0){.92}\put(-.5,.05){\scriptsize$#1$}\put(-1,0){#2}}

\newcommand{\stepwestshiftN}[2]{\put(0,.08){\vector(-1,0){.92}\put(-.5,.05){\scriptsize$#1$}}\put(-1,0){#2}}
\newcommand{\stepwestshiftS}[2]{\put(0,-.08){\vector(-1,0){.92}\put(-.5,-.25){\scriptsize$#1$}}\put(-1,0){#2}}
\title{Shifted tableau crystals} 

\author{
Maria Gillespie \thanks{\href{mailto:mgillespie@math.ucdavis.edu}{mgillespie@math.ucdavis.edu}. Supported by the NSF MSPRF grant PDRF 1604262.}\addressmark{1},
Jake Levinson \thanks{\href{mailto:jlev@uw.edu}{jlev@uw.edu}. Supported by a Rackham Predoctoral Fellowship and by NSERC grant PDF-502633.}\addressmark{2},
\and
Kevin Purbhoo \thanks{\href{mailto:kpurbhoo@uwaterloo.ca}{kpurbhoo@uwaterloo.ca} Supported by NSERC grant RGPIN-355462.}\addressmark{3}
}

\address{
\addressmark{1}Mathematics Department, University of California, Davis, CA \\
\addressmark{2}Mathematics Department,
University of Washington, Seattle, WA \\
\addressmark{3}Combinatorics and Optimization Department, University of Waterloo, ON}

\received{\today}

\abstract{
We introduce coplactic raising and lowering operators $E'_i$, $F'_i$, $E_i$, and $F_i$ on shifted skew semistandard tableaux. We show that the primed operators and unprimed operators each independently form type A Kashiwara crystals (but not Stembridge crystals) on the same underlying set and with the same weight functions. When taken together, the result is a new kind of `doubled crystal' structure that recovers the combinatorics of type B Schubert calculus: the highest-weight elements of our crystals are precisely the shifted Littlewood-Richardson tableaux, and their generating functions are the (skew) Schur $Q$-functions. We give local axioms for these crystals, which closely resemble the Stembridge axioms for type A. Finally, we give a new criterion for such tableaux to be ballot.
}


\usepackage[backend=bibtex]{biblatex}
\begin{document}

\maketitle


\section{Introduction}


  Crystal bases were first introduced by Kashiwara \cite{Kashiwara} in the context of the representation theory of the quantized universal enveloping algebra $U_q(\mathfrak{g})$ of a Lie algebra $\mathfrak{g}$ at $q=0$.  Since then, their connections to tableau combinatorics, symmetric function theory, and other parts of representation theory have made crystal operators and crystal bases the subject of much recent study.  (See \cite{Schilling} for an excellent recent overview of crystal bases.)
  
  In type A, the theory is well-understood in terms of semistandard Young tableaux. There are combinatorial operators $E_i$ and $F_i$ that respectively raise and lower the weight of a tableau by changing an $i+1$ to an $i$ or vice versa.  They are defined in terms of the reading word $w$: one replaces each $i$ in $w$ with a right parenthesis and each $i+1$ with a left parenthesis. Then $E_i(T)$ is formed by maximally pairing parentheses, then changing the first unpaired $i+1$ to $i$. For $F_i(T)$, we instead change the last unpaired $i$ to $i+1$. With $i=1$ and $w=112212112$, this gives:
  \[\varnothing\ \xleftarrow{\ E\ }\ \ ))\underline{(()())})\ \ \xleftarrow{\ E\ }\ w = ))\underline{(()())}(\ \ \xrightarrow{\ F\ }\ \ )(\underline{(()())}(\ \ \xrightarrow{\ F\ }\ \ ((\underline{(()())}(\ \ \xrightarrow{\ F\ }\ \varnothing
  \]
 For tableaux of straight shape and entries in $\{1,2\}$, the action simplifies to the following natural chain structure:
\begin{center}
    \includegraphics{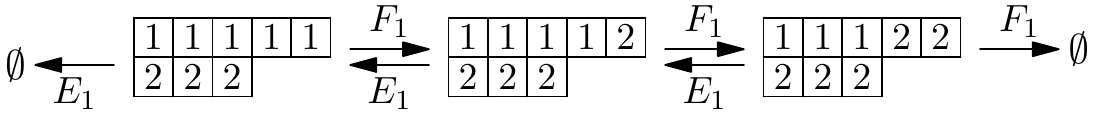}
  \end{center}
  A central property of the $E_i$ and $F_i$ operations is that they are \textbf{coplactic}, that is, they commute with all sequences of jeu de taquin slides. Thus, if we perform the same outwards slide on the tableaux above, the crystal operators must act in the same way:
  \begin{center}
   \includegraphics{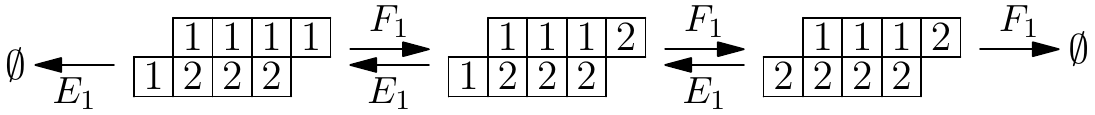}
  \end{center}
In this sense the operations $E_i$ and $F_i$ are the \textit{unique} coplactic operators that give the natural connected chain structure on rectified shapes containing only $i$, $i+1$. Notably, Littlewood-Richardson skew tableaux 
are precisely those for which $E_i(T) = \varnothing$ for all $i$. This gives a proof of the Littlewood-Richardson rule.


\subsection{Shifted tableaux; results of this paper}
Despite the elegance of the crystal operators on ordinary semistandard tableaux, a similar structure on \textbf{shifted tableaux} has proven elusive. In \cite{GJKKK}, Grantcharov, Jung, Kang, Kashiwara, and Kim use Serrano's \emph{semistandard decomposition tableaux} \cite{Serrano} to study the quantum queer superalgebras $\mathfrak{q}_n$ because, ``unfortunately, the set of shifted semistandard Young tableaux of fixed shape does not have a natural crystal structure.'' In this paper, we give a potential resolution: a coplactic crystal-like structure on shifted tableaux.

For straight shifted tableaux on the alphabet $\{1',1,2',2\}$, there is a natural organization by weight of the tableaux of a given shape (Figure \ref{fig:two-row}), similar to the chains for ordinary tableaux. 
Haiman's theory of shifted dual equivalence (see \cite{Haiman}) implies that these operators uniquely extend to coplactic operators on all shifted skew tableaux. However, a direct description that does not rely on shifted jeu de taquin -- like the pairing-parentheses description of $E$ and $F$ on ordinary tableaux -- is far from obvious. 

\begin{figure}[b!]
\begin{center}
 \includegraphics[width=.8\linewidth]{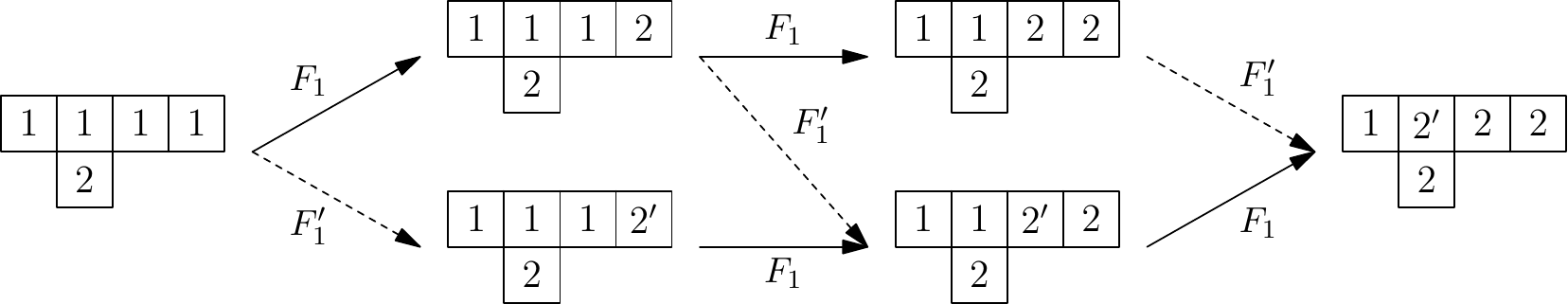} \vspace{0.2cm}
 
 \includegraphics[width=.75\linewidth]{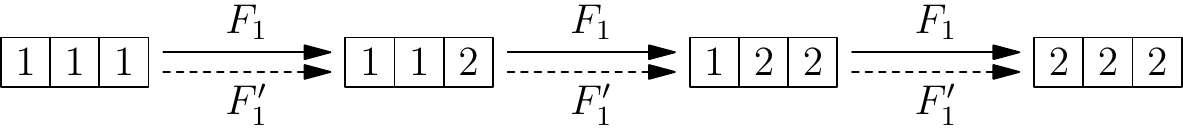}
\end{center}
 \caption{\label{fig:two-row}The crystals of the form $\ShST(\lambda,2)$ are `doubled strings'. {\bf Above}: the tableaux of shapes $(4,1)$ and $(3)$. Wherever an arrow is missing, the corresponding operator is not defined. Reversing the arrows gives the partial inverses $E_1$ and $E_1'$.} 
\end{figure}

There are three main results in this paper. First, we give a direct description of the coplactic operators $E_i,E_i',F_i,F_i'$; see Definitions \ref{def:primed-operators} and \ref{def:F} and Theorem $\ref{thm:main}$. Second, we show that our crystals are uniquely determined by their local structure, and we describe that structure explicitly (Theorem \ref{thm:uniqueness-main}). Third, we obtain a new and simpler criterion for a shifted tableau to be ballot (Theorem \ref{thm:lattice-walk-main}). Two crystals are shown in Figure \ref{fig:crystal}.


\subsection{Applications and future work}

Our primary application (in a forthcoming paper) 
is to understand the topology of so-called {\bf real Schubert curves} $S$ in the odd orthogonal Grassmannian $\mathrm{OG}(n,2n+1)$. (See \cite{GillespieLevinson} for prior work on Schubert curves in $\mathrm{Gr}(k,n)$.) Geometric considerations show that $S$ is described by the operators $F_i$ and shifted jeu de taquin.
%
%
A significant question is whether our crystals also form canonical bases for the representations of some quantized enveloping algebra. 
%
Our local axioms also give a crystal-theoretic way to show that a generating function is Schur-$Q$-positive: one constructs operators on the underlying set, satisfying the local axioms. Schur-$Q$-positivity then follows from Theorem \ref{thm:uniqueness-main}. This method in type A has been used \cite{Morse-Schilling} for certain Stanley symmetric functions.



\begin{remark}
The results of this paper are in \cite{GLP2017}, except for the local axioms of Theorem \ref{thm:uniqueness-main}, which are forthcoming in \cite{GL2018}.
\end{remark}

\subsection{Acknowledgments}

We thank Richard Green, Anne Schilling, David Speyer, John Stembridge and Mark Haiman for helpful conversations. 

\section{Background: words and shifted tableaux}\label{sec:notation}


Let $w$ be a string in symbols $\{1',1,2',2,3',3,\ldots\}$. Informally, we treat the first $i$ or $i'$ in $w$ as both $i$ and $i'$. A {\bf word} is, thus, an equivalence class of strings. The {\bf canonical form} is the string whose first $i$ or $i'$ (for all $i$) is an $i$.
The {\bf weight} of $w$ is $\mathrm{wt}(w) = (n_1, n_2, \ldots ),$ where $n_i$ is the total number of $(i)$s and $(i')$s in $w$.



  We use the conventions of \cite{Sagan, Worley} for shifted tableaux. A \textbf{strict partition} is a strictly-decreasing sequence of positive integers, $\lambda=(\lambda_1 > \ldots > \lambda_k)$. 
  Its \textbf{(shifted) Young diagram} has $\lambda_i$ boxes in the $i$-th row, shifted $i$ steps to the right. A \textbf{(shifted) skew shape}  $\lambda/\mu$ is formed by removing the boxes of $\mu$ from $\lambda$. A \textbf{(shifted) semistandard tableau} $T$ is a filling of $\lambda/\mu$ with letters $\{1'{<}1{<}2'{<}2{<}\cdots \}$, with weakly increasing rows and columns, where primed entries repeat only in columns, and unprimed only in rows.  The \textbf{reading word} $w$ of $T$ is the concatenation of the rows from bottom to top.
The {\bf weight} of $T$ is the weight of its reading word. We write $\B = \ShST(\lambda/\mu,m)$ for the set of shifted semistandard tableaux of shape $\lambda/\mu$ and entries $\leq m$.

  \begin{center}
  \raisebox{-.5\height}{\includegraphics{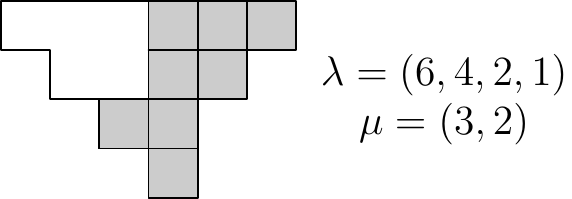}}
\hspace{1cm}
\raisebox{-.5\height}{\includegraphics{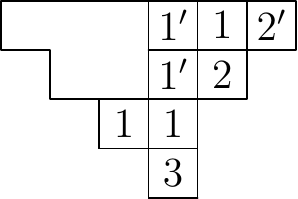}}
\hspace{0.2cm}
 (word: $w = 3111'21'12'$)
  \end{center}

We say that $T$ is in \textbf{canonical form} if $w$ is, and use the same conventions for tableaux as for words: 
two tableaux are \textbf{equivalent} if they have the same shape and their reading words are equivalent; the set of all representatives of $T$ is the \textbf{equivalence class} of $T$.

The notion of \textbf{jeu de taquin} for shifted tableaux is similar to that for usual tableaux: sliding squares while preserving the semistandardness conditions. See \cite{Worley} for details.
%
%
We write $\rectify(T)$ or $\rectify(w)$ for the jeu de taquin \textbf{rectification} of any shifted semistandard tableau $T$ with reading word $w$; this is well-defined by \cite{Sagan, Worley}. We say a tableau $T$ in canonical form is {\bf Littlewood-Richardson}, and $w$ is {\bf ballot}, if for each $i$ the $i$-th row of $\rectify(T)$ contains only $(i)$s. An operation on shifted tableaux (or on their reading words) is \textbf{coplactic} if it commutes with all shifted jeu de taquin slides.



Shifted semistandard tableaux are used to define the Schur $Q$-functions. Equivalence classes of tableaux, on the other hand, arise in formulas involving multiplication of Schur $P$-functions, whose structure coefficients are enumerated by the Littlewood-Richardson shifted tableaux.
Our motivation is connected to the latter, from the geometry of $OG(n,2n+1)$, so we work with equivalence classes of tableaux (or equivalently, tableaux in canonical form). 

  The \textbf{standardization} $\mathrm{std}(w)$ of a word $w$ is formed by replacing its letters by $1,2,\ldots,n$, from least to greatest, breaking ties by reading order for unprimed letters and reverse reading order for primed letters. For a tableau $T$, we form $\std(T)$ by standardizing its reading word; note that $T$ is semistandard if and only if $\mathrm{std}(T)$ is standard.

\section{The operators $E', F'$ and $E,F$}\label{sec:primed-ops}

We now restrict to the alphabet $\{1',1,2',2\}$, and define only $E'_1, F'_1$ and $E_1, F_1$ and write $E',F',E,F$.  In general, the index-$i$ operators $E'_i, F'_i, E_i, F_i$ 
act on the subword of letters $\{i',i,i{+}1',i{+}1\}$, treating $i$ as $1$ and $i{+}1$ as $2$. For tableaux, the operators will act on the reading word.

\subsection{Lattice walks, critical strings and ballotness}

The key construction for our operators is to associate, to a word $w$ in the alphabet $\{1',1,2',2\}$, a first-quadrant {\bf lattice walk}. The walk begins at the origin and converts each letter of $w$ to a unit step in a cardinal direction. See Fig. \ref{fig:lattice-walk-intro}. We think of this as a kind of `bracketing rule', because arrows (away from the axes) `cancel' in opposite pairs. 

\begin{figure}
\begin{center}
\begin{tabular}{cc}
\setlength{\unitlength}{3em}
\ \ \begin{picture}(5,3)(0,0)
\multiput(0,0)(0,0.2){16}{\line(0,1){0.1}}
\multiput(0,0)(0.2,0){17}{\line(1,0){0.1}}
\put(0,1.2){\circle*{0.1}}
\put(0,1.2){\vector(0,1){.7}\vector(1,0){.7}}
\put(-0.265,1.95){\small{$2$,$2'$}}
\put(0.4,1.3){\small{$1$,$1'$}}%
\put(1.4,0){\circle*{0.1}}
\put(1.4,0){\vector(0,1){.7}\vector(1,0){.7}}
\put(1.5,0.6){\small{2,2'}}
\put(2.0,0.1){\small{1,1'}}%
\put(2.8,2.2){\circle*{0.1}}
\put(2.8,2.2){\vector(0,1){.7}\vector(1,0){.7}}
\put(2.8,2.2){\vector(0,-1){.7}}
\put(2.8,2.2){\vector(-1,0){.7}}
\put(2.7,2.95){\small{$2$}}
\put(3.5,2.1){\small{$1'$}}
\put(2.7,1.2){\small{$1$}}
\put(1.85,2.1){\small{$2'$}}
\end{picture}
&
\ \
\begin{picture}(3,3)(0,0)
\setlength{\unitlength}{3em}
\stepnorth{2}{%
\stepeast{1}{%
\stepeast{1'}{%
\stepsouthshiftW{1}{%
\stepnorthA{2'}{%
\stepnorthA{2}{%
\stepwestshiftS{2'}{%
\stepeastA{1'}{%
\stepeastA{1'}{%
}}}}}}}}}
\put(0,0){\circle*{0.15}}
\put(3,2){\circle{0.15}}
\multiput(0,0)(0,0.2){15}{\line(0,1){0.1}}
\multiput(0,-0.02)(0.2,0){15}{\line(1,0){0.1}}
\end{picture}
\end{tabular}
\end{center}
\caption{The {\bf lattice walk} of a word $w \in \{1', 1, 2', 2\}^n$. In the interior of the first quadrant, each letter corresponds to a cardinal direction. Along the axes, primed and unprimed letters behave the same way. \textbf{Right:} The walk for $w=211'12'22'1'1'$ ends at the point $(3,2)$.\label{fig:lattice-walk-intro}} \end{figure}
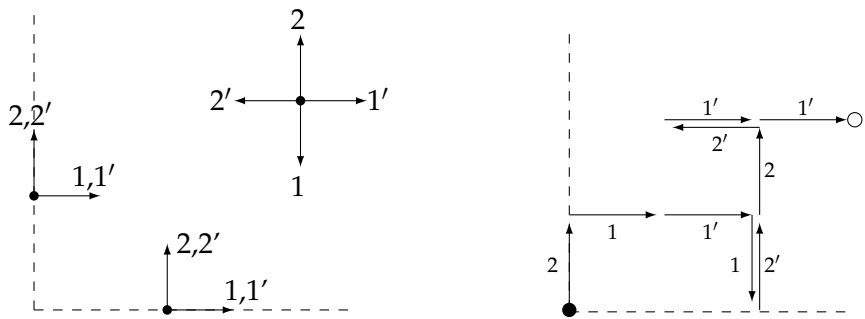


\begin{theorem}
The walk's length $n$ and 
endpoint $(x_n,y_n)$ 
determine the shape 
of $\lambda = \rectify(w)$:
\begin{align*}
\lambda_1 &= \tfrac{1}{2}(n+x_n+y_n) = \#\big\{\ \north{} \text{ and/or } \east{} \text{ steps in the walk\ }\big\}, \\
\lambda_2 &= \tfrac{1}{2}(n-x_n-y_n) = \#\big\{\west{} \text{ and/or } \south{} \text{ steps in the walk\ }\big\}.
\end{align*}
\end{theorem}

\begin{proof} We use shifted Knuth equivalence \cite{Sagan,Worley}, showing that the endpoint of the walk is unchanged after applying a shifted Knuth move to $w$. 
\end{proof}

As a corollary, we obtain a new criterion for ballotness, which differs from existing characterizations (see \cite{Stembridge}) in that it only requires reading through $w$ once, rather than twice (backwards-and-forwards).
\begin{theorem} \label{thm:lattice-walk-main}
Let $w$ be a word in the alphabet $\{1',1, \ldots, n',n\}$. Then $w$ is ballot if and only if each of its lattice walks (for $i=1, \ldots, n-1$) ends on the $x$-axis.
\end{theorem}


\subsection{Definitions of $E', F'$ and $E,F$}

\begin{lemma} \label{lem:unstandardize}
Let $s$ be a standard word (i.e. a permutation of $1, \ldots, n$), and let $n = \sum_{i=1}^k a_i$. There is at most one word $w$ of weight $(a_1,\ldots,a_k)$ and standardization $\mathrm{std}(w) = s$.
\end{lemma}

\noindent This lemma justifies the following definition. Let $\alpha$ be the vector $(1,-1)$.

\begin{definition}[Primed operators] \label{def:primed-operators}
We define $E'(w)$ to be the unique word such that
\[\mathrm{std}(E'(w)) = \mathrm{std}(w) \hspace{0.5cm}\text{ and }\hspace{0.5cm} \mathrm{wt}(E'(w)) = \mathrm{wt}(w) + \alpha,\]
if such a word exists; otherwise, $E'(w) = \varnothing$. We define $F'(w)$ analogously using $-\alpha$.
\end{definition}



%
%

\begin{example} Here are some maximal chains for $F'$:
\begin{gather*}
12211' \xrightarrow{F'} 1222'1' \xrightarrow{F'} \varnothing
\\
1111'1' 
\xrightarrow{F'} 1121'1' 
\xrightarrow{F'} 1221'1' 
\xrightarrow{F'} 22211' 
\xrightarrow{F'} 2222'1
\xrightarrow{F'} 2222'2'
\xrightarrow{F'} \varnothing
\end{gather*}
\end{example}







For the unprimed operators, let $w$ be a word and $u = w_k \dots w_l$ a substring of some representative of $w$. Let $(x,y)$ be the location of the lattice walk of $w$ just before $w_k$.

\begin{definition} We say that $u$ is an \defn{$F$-critical substring} if certain conditions on $u$ and its location are met. There are five types of $F$-critical substring. Each row of the table in Figure \ref{fig:criticals} describes one type, and a transformation that can be performed on that type.
\end{definition} 

\begin{remark} Critical substrings can only occur when the walk has $y=0$, $y=1$, $x=0$, or $x=1$, i.e. when the prefix $w_1 \dots w_{k-1}$ is ballot, anti-ballot, or close to one of these.
\end{remark}

\begin{figure}
\begin{center}
\begin{tabular}{|c|c|c|c|c|}
\hline
\multirow{2}{*}{Type} 
& \multicolumn{3}{c|}{Conditions} & \multirow{2}{*}{Transformation}  \\
     & \multicolumn{1}{c}{Substring}&  \multicolumn{1}{c}{Steps} & Location &  
\\\hline
\hline
 \multirow{2}{*}{1F} & 
\multirow{2}{*}{$u = 1(1')^*2'$} &
\east{1} ~ \east{1'} ~ \north{2'} & 
$y=0$ & 
\multirow{2}{*}{$u \to 2'(1')^*2$} \\[.5ex]\cline{3-4}
& &  
\south{1} ~ \east{1'} ~ \north{2'} & 
$y=1$, $x \geq 1$ & 
\\[.5ex]\hline
\multirow{2}{*}{2F} &
\multirow{2}{*}{$u = 1(2)^*1'$} &
\east{1} ~ \north{2} ~ \east{1'} & 
$x = 0$ & 
\multirow{2}{*}{$u \to 2'(2)^*1$}  \\[.5ex]\cline{3-4}
&& 
\south{1} ~ \north{2} ~ \east{1'} &
$x = 1$, $y \geq 1$ & 
\\[.5ex]\hline
 3F & $u = 1$ & 
\east{1} & 
$y = 0$ & 
$u \to 2$ 
\\\hline
 4F & 
$u = 1'$ & 
\east{1'} & 
$x  = 0$ & 
$u \to 2'$ 
\\\hline
\multirow{2}{*}{5F} & $u = 1$ 
& \south{1}
& \multirow{2}{*}{$x=1$, $y \geq 1$} 
&
\multirow{2}{*}{undefined} \\[.5ex]\cline{2-3}
& $u=2'$ & \west{2'} &&
\\\hline
\end{tabular}
\end{center}
\vspace{0.5cm}

\caption{\label{fig:criticals} $F$-critical substrings and their transformations. Here $a(b)^*c$ means any string of the form $abb \dots bc$, including $ac$, $abc$, $abbc$, etc.}
\end{figure}

We define the {\bf final} $F$-critical substring $u$ of $w$ as follows: we take $u$ with the highest possible starting index, and take the longest in the case of a tie. If there is still a tie (from different representatives of $w$), we take any such $u$.

\begin{definition}[Unprimed operators] \label{def:F}
Let $w$ be a word and $v$ a representative containing the final $F$-critical substring $u$. We define $F(w)$ by transforming $u$ (in $v$) according to its type. If the type is 5F, or if $w$ has no $F$-critical substrings, then $F(w)$ is undefined; we write $F(w) = \varnothing$. The definition of $E$ is analogous, obtained from the definition of $F$ by exchanging $1 \leftrightarrow 2$, primed and unprimed letters, and $x \leftrightarrow y$.
\end{definition}

\begin{example} \label{exa:apply-F}
Let $w=1221'1'111'1'2'2222'2'11'1$.  There are four $F$-critical substrings:
\begin{itemize}
\setlength\itemsep{0em}
\item 
the substring $w_1 = 1$ is critical of type 3F (and 4F in a different representative),
\item the substring $w_1w_2 = 12'$ (in a different representative) is type 1F,
\item the substring $w_1 \cdots w_4 = 1221'$ is type 2F,
\item 
the substring $w_7 \cdots w_{10} = 11'1'2'$ is type 1F.
\end{itemize}
The last is final, so we use rule 1F, $11'1'2' \to 2'1'1'2$, to get $F(w)$. See Figure \ref{fig:F-on-walks}. 

\begin{figure}[t]
\begin{center}
\setlength{\unitlength}{2.5em}
\begin{picture}(5,2)(0,0)
\multiput(0,0)(0,0.2){12}{\line(0,1){0.1}}
\multiput(0,0)(0.2,0){27}{\line(1,0){0.1}}
\put(0,0){\circle*{0.13}}
\put(5,1){\circle{0.13}}
\stepeast{1}{%
\stepnorth{2}{%
\stepnorth{2}{%
\stepeast{1'}{%
\stepeast{1'}{%
\stepsouth{1}{%
\stepsouth{1}{%
\stepeast{1'}{%
\stepeast{1'}{%
\stepnorth{2'}{%
}}}}}}}}}}
\end{picture} \hspace{1cm}
\begin{picture}(4,2)(0,0)
\multiput(0,0)(0,0.2){12}{\line(0,1){0.1}}
\multiput(0,0)(0.2,0){22}{\line(1,0){0.1}}
\put(0,0){\circle*{0.13}}
\put(4,2){\circle{0.13}}
\stepeast{1}{%
\stepnorth{2}{%
\stepnorth{2}{%
\stepeast{1'}{%
\stepeast{1'}{%
\stepsouth{1}{%
\stepwestshiftN{2'}{%
\stepeast{1'}{%
\stepeast{1'}{%
\stepnorth{2}{%
}}}}}}}}}}
\end{picture}
\end{center}
\caption{\label{fig:F-on-walks} 
Computing $F(1221'1'111'1'2')$. Note that the endpoint shifts by $(-1,+1)$.}
\end{figure}
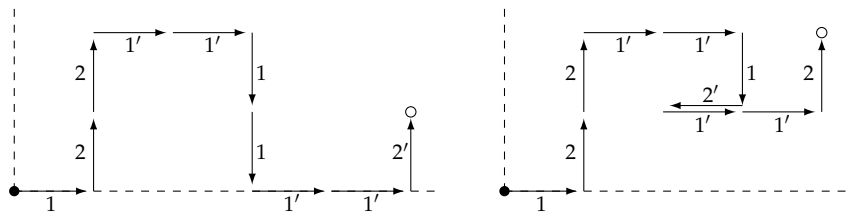
\end{example}


\begin{theorem}[Properties of $E,F$ and $E',F'$]
The operators $E,F$ have the following properties:
\begin{enumerate}
\setlength\itemsep{0em}
\item They are partial inverses, that is, $v = E(w)$ if and only if $F(v) = w$.
\item For tableaux, $E$ and $F$ preserve semistandardness and are coplactic for shifted jeu de taquin.
\item For tableaux of rectified shape, $E$ and $F$ agree with the definitions indicated in Figure \ref{fig:two-row}.
\end{enumerate}
The same properties hold for $E',F'$. 
\end{theorem}
\begin{proof}
The properties of $E',F'$ follow readily from the definitions; the bulk of the work concerns $E,F$.
For (1), we show that the final $F$-critical substring of $E(w)$ is the transformation of the final $E$-critical string of $w$, and that the `tail' of the lattice walk transforms in a simple way.
For (2), we use shifted dual equivalence and mixed insertion to show that $E(w)$ and $F(w)$ are dual equivalent to $w$. It is then fairly straightforward to show that the operators commute with jeu de taquin slides. Verifying (3) is by inspection.
\end{proof}






\subsubsection{Shifted tableau crystals}

  Let $\lambda/\mu$ be a shifted skew shape and let $\B = \ShST(\lambda/\mu,n)$ be the set of shifted semistandard tableaux on the alphabet $\{1'{<}1{<} 2'{<}2{<}\cdots{<}n'{<}n\}.$

\begin{theorem}\label{thm:main}
Let $F'_i,F_i$ ($i = 1, \ldots, n-1$) be the coplactic lowering operators on $\B$ defined in Definitions \ref{def:primed-operators} and \ref{def:F} respectively, with partial inverse (raising) operators $E_i', E_i$. 

\begin{itemize}
\setlength\itemsep{0em}
\item[(i)] The highest-weight elements of $\B$ (those for which $E_i(T) = E_i'(T) = \varnothing$ for all $i$) are precisely the type B Littlewood-Richardson tableaux.
\item[(ii)] 
Each connected component of the induced graph on $\B$ has a \emph{unique} highest-weight element.
\end{itemize}
\end{theorem}
Summing the weights (weighting a vertex of weight $\gamma$ by $2^{\#\{i:\gamma_i> 0\}}x^{\mathrm{wt}(\gamma)}$) and breaking $\B$ into connected components recovers the skew type B Littlewood-Richardson rule,
\[\ShST(\lambda/\mu,n)\ \cong\ \bigsqcup_\nu \ShST(\nu,n)^{f_{\nu,\mu}^\lambda} \qquad \leadsto \qquad Q_{\lambda/\mu} = \sum_\nu f_{\nu,\mu}^\lambda Q_\nu.\]
where $f_{\nu,\mu}^\lambda$ is the coefficient of the Schur $Q$-function $Q_\nu$ in the expansion of the skew Schur $Q$-function $Q_{\lambda/\mu}$.
The crystal also gives an automatic proof of symmetry for $Q_\lambda$.

\begin{figure}[t]
\begin{center}
\includegraphics[height=9cm]{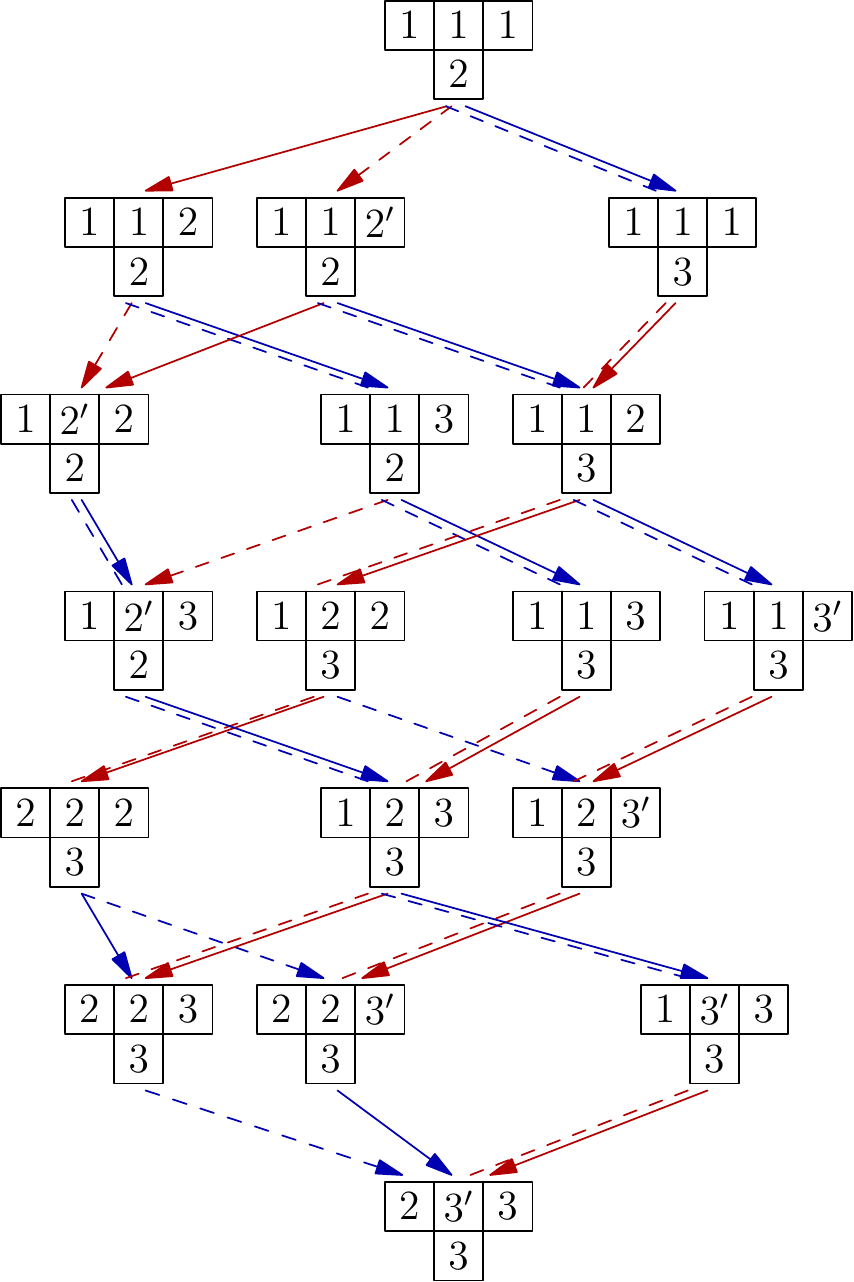} \hspace{1cm} \includegraphics[height=9cm]{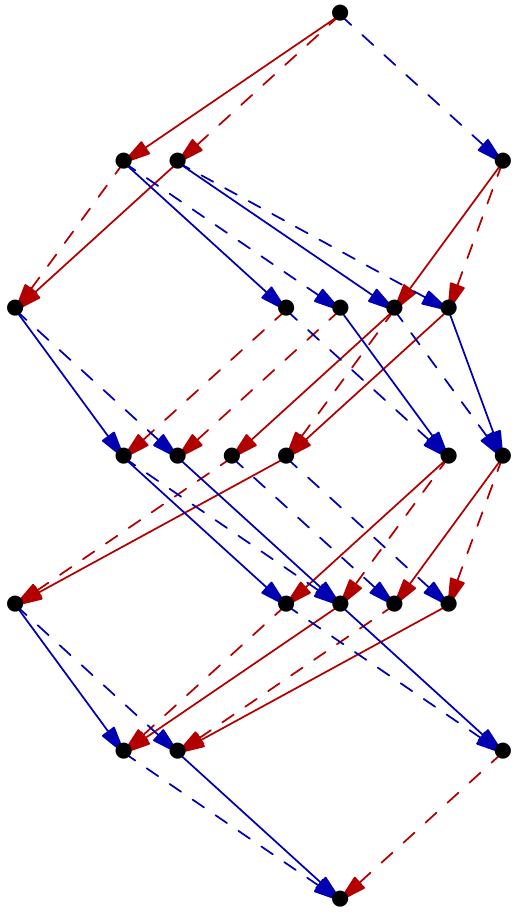}
\end{center}
\caption{\label{fig:crystal} The crystals for $\lambda = (3,1)$ and $(4,2,1)$. The $F_1', F_1$ arrows are, respectively, dashed red and solid red arrows (pointing leftwards), and $F_2', F_2$ are dashed blue and solid blue (pointing rightwards). Clustered vertices have the same weight.
}
\end{figure}

\section{Doubled crystal structure: local axioms}\label{sec:crystal}

Despite its type B enumerative properties, our crystal is not a type B crystal. Instead, it has a `doubled' type A structure, based on the separate actions of the primed and unprimed operators.

Recall that a (finite) {\bf Kashiwara crystal for $\mathrm{GL}_n$} is a set $\B$ together with partial operators $e_i,f_i$ on $\B$, length functions $\varepsilon_i,\varphi_i:\B\to \mathbb{Z}$ for $1 \leq i \leq n-1$, and weight function
$\mathrm{wt}:\B\to \mathbb{Z}^n$, such that:
\begin{enumerate}
\item[(K1)]  The operators $e_i, f_i$ are partial inverses, and if $Y = e_i(X)$, then
\[
(\varepsilon_i(Y), \varphi_i(Y))=(\varepsilon_i(X)-1, \varphi_i(X)+1) \ \ \text{ and } \ \
\mathrm{wt}(Y) = \mathrm{wt}(X)+\alpha_i,
\]
where $\alpha_i = (0,\ldots,1,-1,\ldots,0)$ is the weight vector with $1,-1$ in positions $i,i{+}1$.

\item[(K2)] For any $i\in \{1,\ldots,n-1\}$ and any $X\in \B$, we have $\varphi_i(X)=\langle \mathrm{wt}(X),\alpha_i\rangle+\varepsilon_i(X)$.
\end{enumerate}

Let $\B = \ShST(\lambda/\mu, n)$. We have two overlapping Kashiwara crystals on $\B$, given by our operators as follows.

\begin{definition}\label{def:crystal}
For $1 \leq i \leq n-1$ and $T \in \B$, let $(\varphi_i(T),\varepsilon_i(T)) := (x^{(i)},y^{(i)})$
be the endpoint of the $i,i+1$ lattice walk associated to $T$. Let $\mathrm{wt}(T)$ be its weight as a tableau . \end{definition}


\begin{theorem}\label{thm:main-doubled-typeA}
The operators $F_i,E_i,F'_i,E'_i$ commute whenever compositions are defined. Moreover, $F_i,E_i$ and $F'_i,E'_i$  independently satisfy the type A Kashiwara crystal axioms, using the same auxiliary functions $\varepsilon_i,\varphi_i,\wt$  on $\B = \ShST(\lambda/\mu,n)$. We call $\B$ a \defn{shifted tableau crystal}. Two examples are shown in Figure \ref{fig:crystal}.
\end{theorem}

We note that our crystals do not satisfy the Stembridge axioms \cite{Stembridge-typeA}, and therefore are not crystals in the sense of the representation theory of $U_q(\mathfrak{sl}_n)$. Instead, we show that they satisfy a `doubled' form of these axioms.

Explicitly, we show that the crystals $\B = \ShST(\lambda,n)$ are determined by their local combinatorial structure -- specifically the interactions between the $i,i',j,j'$ operators. The analogous statement for ordinary (non-shifted) tableaux is due to Stembridge \cite{Stembridge-typeA}. We show the following universality statement:



\begin{theorem}
\label{thm:uniqueness-main}
Let $G$ be a finite $\mathbb{Z}_{\geq 0}^n$-weighted directed graph with edges labeled $i',i$ for $i=1, \ldots, n-1$. Suppose $G$ satisfies the Kashiwara axioms (K1) and (K2) (with operators $e_i, f_i$ and $e_i', f_i'$, and length functions $\varepsilon_i,\varphi_i$, given by Definition \ref{def:stats}) in addition to the axioms (A1)-(A5) given below.

Then each connected component $C$ of $G$ has a unique highest-weight element $g^*$, with weight $\lambda = \mathrm{wt}(g^*)$ a strict partition, and canonically $C \cong \ShST(\lambda,n)$.
\end{theorem}

In particular, the generating function of $G$ (weighting a vertex $v$ by $2^{\#\{i:\mathrm{wt}(v)_i> 0\}}$) is Schur-$Q$-positive. Thus, any graph $G$ that is locally (shifted-)crystal-like is globally a shifted tableau crystal.

\label{sec:uniqueness}

We now state the axioms. Let $G$ be a graph as in Theorem \ref{thm:uniqueness-main}. 
Below, an omitted edge label is either unprimed (if the edge is solid) or primed (if the edge is dashed). Consider the following axioms on $G$:

\begin{itemize}
\item[(A1)] Each $\{i,i'\}$-connected component is a `doubled string':
\begin{center}
\raisebox{-.5\height}{\includegraphics[scale=1]{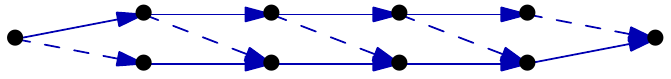}} \ or \ 
\raisebox{-.5\height}{\includegraphics[scale=1]{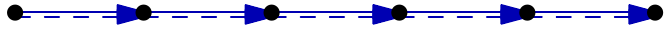}}
\end{center}
Moreover, $v$ is the top (resp. bottom) of a doubled string of the second kind if and only if $\mathrm{wt}_{i{+}1}(v) = 0$ (resp. $\mathrm{wt}_i(v) = 0$).  The smallest possible string of the first kind has two vertices, connected by only an $i'$ arrow.
\item[(A2)] If $|i-j| > 1$, all edges and reverse-edges commute. That is, given two edges $w \xrightarrow{a} x, w \xrightarrow{b} y$, there is a vertex $z$ with edges $x \xrightarrow{b} z, y \xrightarrow{a} z$, and conversely.
\end{itemize}

From (A1), we define $f_i(v)$ and $f_i'(v)$, for $v \in G$, by following the unique $i$ or $i'$ edge from $v$, if it exists, and $\varnothing$ otherwise.  We define $e_i, e_i'$ as the (partial) inverse operations. 
%
We also define statistics $\varepsilon_i, \varphi_i$ on $G$, and require that they satisfy (K1) and (K2) independently using both $e_i, f_i$ and $e_i', f_i'$: 

\begin{definition}\label{def:stats}
We let $\varepsilon_i(v)$ and $\varphi_i(v)$ be the total distance from $v$ to the top and bottom of its doubled string. We similarly define $\varepsilon_i'(v), \varphi_i'(v)$ (counting primed edges only) and $\widehat{\varepsilon}_i(v), \widehat{\varphi}_i(v)$ (counting unprimed edges only).
\end{definition}

\begin{itemize}


\item[(A3)]
Suppose $w\ {\xrightarrow{\ i{\pm}1 \text{ or } i{\pm}1' \ }}\ x$. Then $(\varepsilon_i(w) - \varepsilon_i(x),\varphi_i(w) - \varphi_i(x)) = (1,0) \text{ or } (0,-1).$ \\
That is, the $i$-string either shortens by 1 at the top, or lengthens by 1 at the bottom:
\begin{center}
\raisebox{-.5\height}{\includegraphics[height=1.5cm]{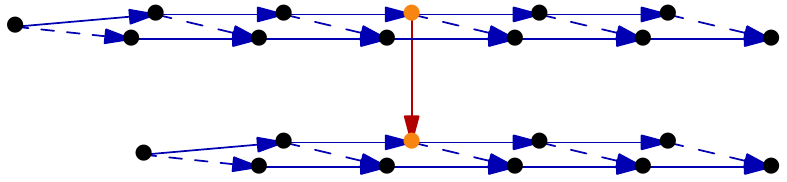}} \ or \ \ 
\raisebox{-.5\height}{\includegraphics[height=1.5cm]{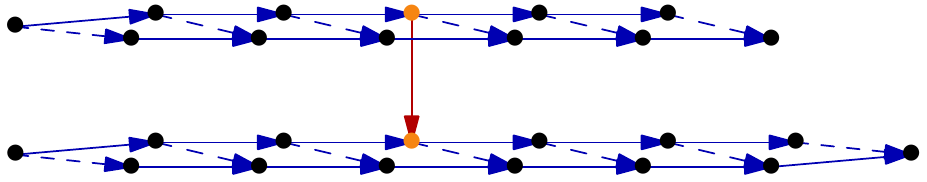}}
\end{center}
\end{itemize}

\noindent For the remaining axioms, we set $i=1$ and suppose there are edges $w \xrightarrow{1' \text{ or } 1} x$, $w \xrightarrow{2' \text{ or } 2} y$. We describe the relations satisfied by $f_1, f_1', f_2, f_2'$ in each case.  The relations for $\{i',i,i{+}1',i{+}1\}$ are the same, obtained by treating $i$ as 1 and $i{+}1$ as 2. Let \[\Delta = (\varepsilon_2(w) - \varepsilon_2(x), \varepsilon_1(w) - \varepsilon_1(y)) = (0,0), (1,0), (0,1) \text{ or } (1,1).\]

Below, we draw $f_1, f_1'$ edges pointing left and $f_2, f_2'$ edges pointing right.
\begin{itemize}
\item[(A4.1)] The following table describes the interactions for the pairs $\{f_1', f_2'\}$, $\{f_1, f_2'\}$ and $\{f_1', f_2\}$. For the $\{f_1', f_2\}$ pair, we additionally assume $f_2(w) \ne f_2'(w)$. In each case, the stated conditions hold if and only if the two top arrows merge as shown. 
\begin{center}
\begin{tabular}[c]{|@{}c@{}|@{}c@{}|c||@{}c@{}|@{}c@{}|c|} \hline
Pair & Conditions & Axiom & Pair & Conditions & Axiom \\ \hline
$\ \{f_1', f_2'\}\ $ & -- & \raisebox{-.5\height}{\includegraphics[scale=0.8]{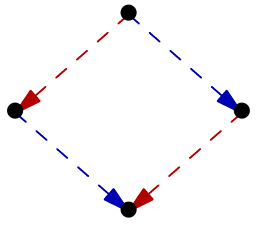}} &
$\ \{f_1', f_2\}\ $ & 
-- & \raisebox{-.5\height}{\includegraphics[scale=0.8]{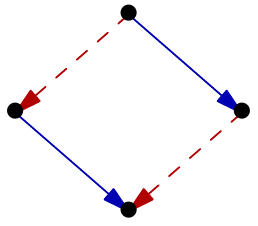}} \\ \hline
$\ \{f_1', f_2'\}\ $ &
\begin{tabular}[c]{c}
$\Delta = (0,0)$, \\
$\varphi_2(w) = 1$, \\
$\widehat{\varphi}_2(w) = 0$
\end{tabular}&
\raisebox{-.5\height}{\includegraphics[scale=0.8]{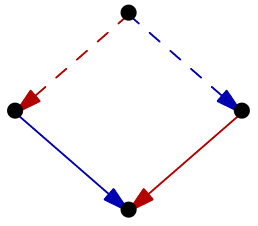}} &
$\ \{f_1, f_2'\}\ $ &
\ \ $\widehat{\epsilon}_1(w) > 0$ \ \ & \raisebox{-.5\height}{\includegraphics[scale=0.8]{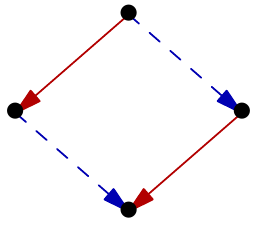}} \\ \hline
\end{tabular} \\
\end{center}

For example, the upper right axiom says that if $f_1'(w)$ and $f_2(w)$ are both defined and $f_2(w)\neq f_2'(w)$ (as assumed for this case), then $f_1'f_2(w)=f_2f_1'(w)$.  

The lower right says that if $f_1(w)$ and $f_2'(w)$ are both defined, then $f_1f_2'(w)=f_2'f_1(w)$ if and only if $\widehat{e}_1(w)>0$.

\item[(A4.2)] For the pair $\{f_1, f_2\}$, we assume $f_1'(w) = \varnothing$. Then the possible relations are:
\begin{center}
\begin{tabular}[c]{|c|c||@{}c@{}|c|} \hline
Conditions & Axiom & Conditions & Axiom \\ \hline
$\Delta = (0,1)$ or $(1,0)$
 & \raisebox{-.5\height}{\includegraphics[scale=0.8]{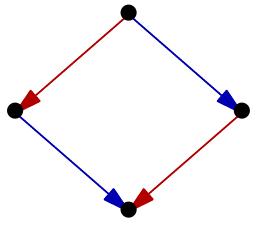}}  &
$\Delta = (1,1)$ & \raisebox{-.5\height}{\includegraphics[scale=0.8]{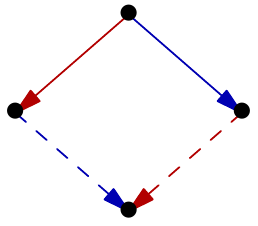}} 
\\ \hline
\begin{tabular}[c]{c}
$\Delta = (0,0)$, \\
$\widehat{\varphi}_1(y) \geq 2$
\end{tabular}&
\raisebox{-.5\height}{\includegraphics[scale=0.8]{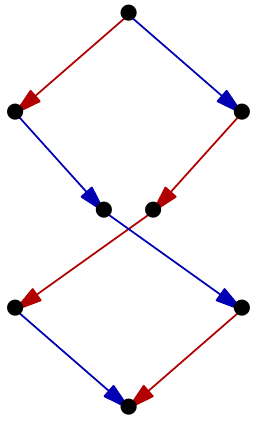}} &
\begin{tabular}[c]{c}
$\Delta = (0,0)$, \\
$\widehat{\epsilon}_1(w) \ne \widehat{\epsilon}_1(y)$
\end{tabular}& \raisebox{-.5\height}{\includegraphics[scale=0.8]{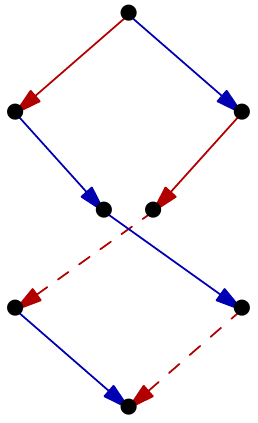}} \\
\hline
\end{tabular}
\end{center}

Note that if two vertices of the same weight are drawn as distinct above, the axiom asserts that they are not equal.  For instance, the lower left axiom says that if $f_1(w)$ and $f_2(w)$ are both defined and $f_1'(w)=\varnothing$, then $f_1f_2^2f_1(w)=f_2f_1^2f_2(w)$ and $f_1f_2(w)\neq f_2f_1(w)$ if and only if $\Delta=(0,0)$ and $\widehat{\varphi}_1(y)\ge 2$.

\item[(A5)] (Duality) We assume `dual' axioms to (A4.1) and (A4.2), reversing edge directions, the labels $1 \leftrightarrow 2$ (but not primed-ness of edges), and the statistics $\varepsilon, \varepsilon', \widehat{\varepsilon} \leftrightarrow \varphi, \varphi', \widehat{\varphi}$.
\end{itemize}

We highlight the similarity between our axioms and the type A Stembridge axioms, apart from the `doubling' effect of the dashed edges (i.e. primed operators).


\begin{proof}[Proof of Theorem \ref{thm:uniqueness-main}]
We show the axioms hold for $G = \ShST(\lambda,n)$ directly. For general $G$, we may assume $G$ is connected. We then show that $G$ has a unique maximal element. Finally, we construct the isomorphism inductively from the top.
\end{proof}

\printbibliography

\end{document}